\documentclass[10pt,a4paper,twoside,notitlepage]{amsart}
\usepackage[utf8]{inputenc}
\usepackage{graphicx}
\usepackage{preambolo}

\begin{document}

\thispagestyle{empty}
\begin{center}
\begin{large}\textbf{DOLBEAULT AND $J$-INVARIANT COHOMOLOGIES ON ALMOST COMPLEX MANIFOLDS}\end{large} \\ \hspace{0.2cm} \\
\begin{small}
\textsc{Lorenzo Sillari and Adriano Tomassini}
\end{small}
\begin{abstract} 
\noindent \textsc{Abstract}. In this paper we relate the cohomology of $J$-invariant forms to the Dolbeault cohomology of an almost complex manifold. We find necessary and sufficient condition for the inclusion of the former into the latter to be true up to isomorphism. We also extend some results obtained by J. Cirici and S.O. Wilson about the computation of the left-invariant cohomology of nilmanifolds to the setting of solvmanifolds. Several examples are given. 
\end{abstract}
\end{center}

\blfootnote{\hspace{-0.55cm} 
2010 \textit{Mathematics Subject Classification}. 32Q60, 53C15. \\ 
\textit{Keywords} Almost complex manifold; cohomology of Lie Algebra; compact four-manifold; Dolbeault cohomology; Fr\"olicher spectral sequence; solvmanifold.\\[5pt]
Partially supported by the
Project PRIN 2017 “Real and Complex Manifolds: Topology, Geometry and holomorphic dynamics” (code 2017JZ2SW5) and by
GNSAGA of INdAM.}
\vspace*{1.5em}

\section{Introduction}

Let $(M,J)$ be a $2m$-dimensional almost complex manifold. Then the almost complex structure $J$ induces a bigrading on the bundle of differential forms on $M$. The exterior derivative $d$ acts on differential forms as the sum of four differential operators, $d= \mu + \partial + \bar \partial + \bar \mu$. 
\newline
The celebrated theorem of Newlander and Nirenberg states that $M$ admits the structure of complex manifold, i.e., $J$ is integrable, if and only if $N_J=0$, that is equivalent to $\mu = \bar \mu=0$. Consequently, in such a case $d = \partial + \bar \partial$. For complex manifolds it is classical and well established the theory of Dolbeault cohomology, obtained as the cohomology of the $\bar \partial$ operator. Another fundamental tool is the Hodge Theory for the $\bar \partial$ operator that, once fixed a Hermitian metric, establishes an isomorphism between the Dolbeault cohomology and the kernel of the Dolbeault Laplacian $\Delta_{\bar \partial}$. However for almost complex manifolds, the operator $\bar \partial$ is not cohomological and the Dolbeault cohomology cannot have the usual definition. It is natural to look for other cohomological theories to study geometric properties of almost complex manifolds.
Motivated by the comparison between the $J$-{\em tamed symplectic cone} $
\mathcal{K}^t_J$ and the $J$-{\em compatible symplectic cone} $
\mathcal{K}^c_J$ of an almost complex manifold, defined as the projection in cohomology of the space of symplectic forms taming $J$, respectively calibrating $J$, 
Li and Zhang introduced in \cite{li:zhang} the $J$-{\em invariant cohomology}, respectively $J$-{\em anti-invariant cohomology} groups of an almost complex manifold $(M,J)$, denoted with $H^+$, respectively $H^-$, formed by $2^{nd}$-de Rham classes represented by closed $J$-invariant, respectively $J$-anti-invariant forms, with respect to the natural action of $J$ on the space of $2$-forms.

Such groups generalize the real Dolbeault cohomology classes in $H^{1,1}_{\bar \partial} \cap H^2_{dR}(\R)$ and $H^{2,0}_{\bar \partial} + H^{0,2}_{\bar \partial} \cap H^2_{dR}(\R)$ respectively. The focus is on whether the almost complex structure $J$ is $C^\infty$-{\em pure}, i.e., $H^+ \cap H^- = \{ 0 \}$ or $C^\infty$-{\em full}, i.e., $H^2_{dR} = H^+ + H^- $. The problem is further studied in \cite{DLZ:4manifolds}, where it is proved that any almost complex structure on a compact $4$-manifold is $C^\infty$-pure and $C^\infty$-full, and in \cite{DLZ:antiinvariant}. Such a result can be viewed as a sort of Hodge decomposition for $4$-dimensional compact almost complex manifolds.
\vspace{.2cm}
\newline

Recently J. Cirici and S. O. Wilson defined in \cite{CW:dolbeaultcohomology} an analogous of Dolbeault cohomology for almost complex manifolds, that is also called Dolbeault cohomology. This idea of cohomology is based on the decomposition of $d$ and allows a development of a harmonic theory, at least in some favorable situation such as in \cite{CW:kahler} for the almost K\"ahler case (see also \cite{TT20}). A Fr\"olicher spectral sequence $E^{p,q}_r$ builds a bridge between the Dolbeault cohomology and the complex de Rham cohomology. In general, the computation of such groups is difficult, since they might not be finite-dimensional. A special setting in which calculations can be performed is that of Lie Algebra. Such computations have a direct application in the study of the left-invariant Dolbeault cohomology of nilmanifolds, as showed in \cite{CW:dolbeaultcohomology}. 
\vspace{.2cm}
\newline
In this paper we study the relation between the complex cohomology group $H^+_\C$ of $J$-invariant complex forms and the Dolbeault cohomology group $H^{1,1}_{Dol}$ on almost complex manifolds. Next we extend some results obtained in \cite{CW:dolbeaultcohomology} for nilmanifolds, to the case of solvmanifolds. More in details, since we have a characterization of $J$-invariant $2$-forms as real forms of complex bidegree $(1,1)$, it is natural to ask whether they belong or not to the Dolbeault cohomology groups, or at least if there exists an isomorphism between $H^+$ and a subgroup of $H^{1,1}_{Dol}$. We relate $J$-invariant cohomology and Dolbeault cohomology, finding that the condition
\[
E^{0,1}_1 \cong E^{0,1}_2
\]
is necessary and sufficient for the former cohomology group to be contained into the latter up to isomorphism (Theorem \ref{theo:inclusion:iso}). Then given any solvmanifold endowed with a left-invariant almost complex structure, we prove that the left-invariant spectral sequence satisfies Serre duality at every stage and that the left-invariant Dolbeault cohomology groups are isomorphic to the kernel of a suitable Laplacian (Theorem \ref{main:solvmanifold}). Finally calculations of left-invariant spectral sequence and $J$-invariant cohomology are performed on almost complex manifolds and solvmanifolds endowed with a left-invariant almost complex structure to give concrete applications.

The paper is organized as follows. In section $2$ we briefly recall some basic definitions that will be used later on, and we introduce the notation. In section 3 we resume the definition given by Cirici and Wilson of Dolbeault cohomology for almost complex manifolds. In particular, we focus on the spectral sequence arising from a Hodge filtration, and give an explicit description of it. Section 4 is devoted to the study of $J$-invariant cohomology and Dolbeault cohomology. We prove the results mentioned above, and investigate the behaviour of the necessary and sufficient condition under small deformations, proving with an example that it is not a closed property. Section 5 recalls the construction of the Dolbeault cohomology of Lie Algebras, while in section 6 we prove the Serre duality for solvmanifolds. Finally in section 7 we collect various examples of Dolbeault cohomology and spectral sequence. Among them, we provide computations of the left-invariant spectral sequence on $4$-dimensional solvmanifolds that do not admit any integrable almost complex structure. For such examples the Dolbeault cohomology theory for almost complex manifolds becomes the main tool to investigate their geometry.
\medskip

\textit{Acknowledgements.} The authors would like to thank Joana Cirici and Weiyi Zhang for useful comments and remarks.

\section{Preliminaries and notation}

Let $(M, J)$ be an almost complex manifold of real dimension $2m$, with $J$ an almost complex structure on the tangent bundle, i.e., $ J \in End (TM)$ such that $J^2 = - Id$. Denote by $A^*_\R$ (respectively $A^*_\C$)  the algebras of real (respectively complex) differential forms on $M$. $J$ induces a bigrading on complex forms,
\begin{equation}
A^k_\C = \bigoplus_{p+q=k} A^{p,q}_\C.
\end{equation}
On real $k$-forms, $\alpha \in A^k_\C$, $J$ induces a map still denoted by $J$ and defined as
\begin{equation}
    J \alpha (X_1, \dots , X_k) = \alpha (JX_1, \dots, JX_k).
\end{equation}
If $k$ is odd, $J^2=-Id$, while if $k$ is even, $J$ is an involution. In particular, $A^2_\R$ decomposes as 
\begin{equation}
    A^2_\R = A^+_\R + A^-_\R,
\end{equation}
where $A^+_\R$ denotes the $J$-invariant forms and $A^-_\R$ the $J$-anti-invariant forms. If we consider the bidegree induced on complex forms by $J$, it's easy to check that $A^+_\R$ consists of real forms in $A^{1,1}_\C$, while $A^-_\R$ of real forms in $A^{2,0}_\C + A^{0,2}_\C$. We denote with $H^{*}_{dR}(\R)$ (respectively $H^{*}_{dR}(\C)$) the real (respectively complex) de Rham cohomology of $M$. 
\newline 
The de Rham cohomology groups consisting of $J$-invariant and $J$-anti-invariant forms were introduced in \cite{li:zhang}. We shall use the notation of \cite{DLZ:antiinvariant}. The \emph{$J$-invariant} real cohomology group is
\begin{equation}
 H^+ = \Big \{ [\alpha] \in H^2_{dR} (\R) : \alpha \in A^+_\R \cap \ker d \Big \},
\end{equation}
and the \emph{$J$-anti-invariant} real cohomology group is 
\begin{equation}
    H^- = \Big \{ [\alpha] \in H^2_{dR} (\R): \alpha \in A^-_\R \cap \ker d \Big \}.
\end{equation}
In the following we will denote a $(p,q)$-form $\alpha$ with $\alpha^{p,q}$.

\vspace{.3cm}
We call a \emph{solvmanifold} the quotient of a connected, simply connected and solvable Lie Group $G$, by a discrete and co-compact subgroup $\Gamma$ of $G$. We denote it by $\Gamma \backslash G$. If $G$ is also nilpotent, we call $\Gamma \backslash G$ a \emph{nilmanifold}.

\section{Dolbeault cohomology and spectral sequences}

Let $(M,J)$ be an almost complex manifold of real dimension $2m$. The exterior derivative decomposes as $d= \mu + \partial + \bar \partial + \bar \mu$, with bidegrees
\begin{equation}
    \abs{\mu} = (2,-1), \quad \abs{\partial} = (1,0), \quad \abs{\bar \partial} = (0,1), \quad \abs{\bar \mu}=(-1,2).
\end{equation}
The almost complex structure $J$ is integrable if and only if $\bar \mu \equiv 0$. The equation $d^2=0$ gives the relations
\begin{equation}\label{multidiff:equations}
\tag{$\square$}
\begin{cases}
\mu^2 =0 \\
\mu \partial + \partial \mu =0 \\
\mu \bar \partial + \bar \partial \mu + \partial^2 =0 \\
\mu \bar \mu + \partial \bar \partial + \bar \partial \partial + \bar \mu \mu =0 \\
\bar \mu \partial + \partial \bar \mu + \bar \partial^2 =0 \\
\bar \mu \bar \partial + \bar \partial \bar \mu =0 \\
\bar \mu^2=0
\end{cases}
\end{equation}
Since $\bar \mu^2=0$, $\bar \mu$ is a well defined cohomological operator and its cohomology is the $\bar \mu$-cohomology
\begin{equation}
    H^{p,q}_{\bar \mu} = \frac{\ker( \bar \mu \colon A^{p,q}_\C \longrightarrow A^{p-1,q+2}_\C)}{\Ima ( \bar \mu \colon A^{p+1,q-2}_\C \longrightarrow A^{p,q}_\C)}.
\end{equation}
In general, $\bar \partial$ does not square to $0$ on $M$, and its cohomology is not well defined. From the relation $\bar \partial \bar \mu + \bar \mu \bar \partial=0$, $\bar \partial$ is well defined on cohomology classes of $H^{p,q}_{\bar \mu}$, and thanks to $\bar \partial^2 + \bar \mu \partial + \partial \bar \mu=0$, it squares to $0$, thus we can define the \emph{Dolbeault cohomology} of the almost complex manifold $(M,J)$ as the $\bar \partial$-cohomology of the $\bar \mu$-cohomology, i.e.,
\begin{equation}
    H^{p,q}_{Dol} = \frac{\ker( \bar \partial \colon H^{p,q}_{\bar \mu} \longrightarrow H^{p,q+1}_{\bar \mu})}{\Ima ( \bar \partial \colon H^{p,q-1}_{\bar \mu} \longrightarrow H^{p,q}_{\bar \mu})}.
\end{equation}

The cohomology groups are well defined and if $\bar \mu=0$, they coincide with the usual Dolbeault cohomology groups for complex manifolds. As in the complex case, the Dolbeault cohomology is induced by a filtration on differential forms and has an associated spectral sequence that coincide with the Fr\"olicher spectral sequence of complex manifolds if $\bar \mu=0$. We recall here the construction: consider the filtration 
\begin{equation}
 F^p A^k_\C =   A^{p,q}_\C \cap \ker \bar \mu \oplus \bigoplus_{j \ge p+1} A^{j, k-j}_\C .
 \end{equation}
The filtration is bounded by 0 from below and by $m$ from above. With a shift of indexing, the filtration coincides with the filtration of a suitable multicomplex endowed with $4$ differentials, and the two spectral sequences are isomorphic, up to taking stages shifted by one step. Denote with 
\begin{equation}
    E^{p,q}_r, \qquad \text{$p$, $q = 0, \dots, m$,  $r \ge 1$},
\end{equation}
the stages of the sequence. Then $E^{p,q}_1 \cong H^{p,q}_{Dol}$ and the $(r+1)$-th stage is the cohomology of the previous one,
\begin{equation}
E^{p,q}_{r+1} \cong \frac{\ker ( d_r \colon E^{p,q}_r \longrightarrow E^{p+r, q-r+1}_r) }{\Ima ( d_r \colon E^{p-r,q+r-1}_r \longrightarrow E^{p,q}_r},
\end{equation}
with respect to the differential $d_r$, $\abs{d_r}=(r, -r+1)$. An explicit description up to isomorphism of the differential $d_r$ and of the stages of the spectral sequence is given in \cite{LWZ:multicomplex} for a general multicomplex, and was first described in \cite{Cordero} for the Fr\"olicher spectral sequence of a complex manifold. For stage 1, we have 
\begin{equation}
    E^{p,q}_1 \cong \frac{\{\alpha \in A^{p,q}_\C \cap \ker \bar \mu: \bar \partial \alpha\in \Ima \bar \mu \}}{\{ \eta \in A^{p,q}_\C: \eta = \bar \mu a + \bar \partial b \text{ and }\bar \mu b =0 \}},
\end{equation}
and if $\bar \partial \alpha = \bar \mu \varphi$, then
\begin{equation}
    d_1[ \alpha]_{E_1} = [\partial \alpha - \bar \partial \varphi]_{E_1}.
\end{equation}
In general, we have an isomorphism 
\begin{equation}\label{isomorphism:spectral}
E^{p,q}_r \cong \frac{X^{p,q}_r}{Y^{p,q}_r},
\end{equation}
where
\begin{multline}\label{numerator}
    X^{p,q}_r = \Big \{ \alpha^{p,q} \in A^{p,q}_\C : \text{ there exist } \alpha^{p+j,q-j}, \, \, j =1, \dots r,  \text{ satisfying } \\ 
    0 = \mu \alpha^{p+k,q-k} + \partial \alpha^{p+k+1,q-k-1} + \bar \partial \alpha^{p+k+2,q-k-2} + \bar \mu \alpha^{p+k+3,q-k-3} \quad k \in \Z \Big \}
\end{multline}
where the equation has to be read with $\alpha^{p+k,q-k}=0$ if $k$ is not $0,1, \dots, r$, and 
\begin{multline}\label{denominator}
    Y^{p,q}_r = \Big\{ \eta^{p,q} \in A^{p,q}_C: \text{ there exist } \eta^{p-j,q+j-1}, \, \, j=-1,0,\dots, r-1 \text{ satisfying } \\
    \eta^{p,q} = \mu \eta^{p-2,q+1} + \partial \eta^{p-1,q} + \bar \partial \eta^{p,q-1} + \bar \mu \eta^{p+1,q-2}, \\
    0 = \mu \eta^{p-k,q+k-1} + \partial \eta^{p-k+1,q+k-2} + \bar \partial \eta^{p-k+2,q+k-3} + \bar \mu \eta^{p-k+3,q+k-4} \quad k=3,\dots, r-1
    \Big\}.
\end{multline}

We say that the spectral sequence degenerates at stage $r$, for bidegree $(p,q)$, and write $E^{p,q}_r \cong E^{p,q}_\infty$, if 
\begin{equation}
    E^{p,q}_r \cong E^{p,q}_ j \qquad \forall j \ge r.
\end{equation} 
The spectral sequence degenerates at stage $r$ if it degenerates at stage $r$ for all bidegrees. 
\newline
At the $E_\infty$ stage, the degeneration of the spectral sequence induces a bigrading on the de Rham cohomology of the almost complex manifold. In particular, $E^{p,q}_\infty$ represents cohomology classes in $H^{p+q}_{dR} (\C)$ that admit a complex representative of bidegree $(p,q)$.

\section{Inclusion of \emph{J}-Invariant cohomology into Dolbeault cohomology}\label{contained}
Denote by $H^+_\C = H^+ \otimes \C$ the complexified of the $J$-invariant cohomology group. We are going to study under which conditions $H^+_\C$ is isomorphic to a subgroup of $H^{1,1}_{Dol}$ through the isomorphism defined in \eqref{isomorphism:spectral} between the Fr\"olicher spectral sequence and the quotients $X^{p,q}_r / Y^{p,q}_r$.

The results are stated in Theorem \ref{theo:inclusion:iso}, that gives a characterization valid in the almost complex case. At the end, we briefly investigate the stability of the condition under small deformations (in the integrable case).
\vspace{.3cm}

For almost complex manifolds of any dimension, $H^+_\C$ consists of complex de Rham cohomology classes in $H^2_{dR}(\C)$ that admit a representative of type $(1,1)$ (cf. \cite{DLZ:4manifolds}, Lemma 2.11). First of all, we could ask wether the identity on the representatives of classes in $H^+_\C$ induces the inclusion 
\begin{equation}\label{inclusion}
    H^+_\C \subseteq H^{1,1}_{Dol}.
\end{equation}
This is clearly false since any given $[\alpha] \in H^+_\C$ is given by
\[
[\alpha] = \{ \alpha^{1,1} + d (\beta^{1,0} + \beta^{0,1}) : \, d \alpha^{1,1}=0 \}
\]
and the class admits representatives that are not of pure bidegree $(1,1)$, while all the representatives in $H^{1,1}_{Dol}$ are of pure bidegree $(1,1)$. However this is the case up to the above mentioned isomorphism. We make use of the explicit description of the stages of the spectral sequence recalled in \eqref{numerator} and \eqref{denominator}. 

At bidegree $(1,1)$, the $r$-th stage of the spectral sequence is $E^{1,1}_r \cong X^{1,1}_r / Y^{1,1}_r$, where 
\begin{align*}
    X^{1,1}_1 &= \{ \alpha^{1,1}: 0= \bar \mu \alpha^{1,1} = \bar \partial \alpha^{1,1} + \bar \mu \alpha^{2,0} \}, \\
    X^{1,1}_2 &= \{ \alpha^{1,1}: 0= \bar \mu \alpha^{1,1} = \bar \partial \alpha^{1,1} + \bar \mu \alpha^{2,0} = \partial \alpha^{1,1} + \bar \partial \alpha^{2,0} \}, \\
    X^{1,1}_2 &= \{ \alpha^{1,1}: 0= \bar \mu \alpha^{1,1} = \bar \partial \alpha^{1,1} + \bar \mu \alpha^{2,0} = \partial \alpha^{1,1} + \bar \partial \alpha^{2,0} = \mu \alpha^{1,1} + \partial \alpha^{2,0} \}, \\
    Y^{1,1}_1 &= \{ \eta^{1,1} = \bar \partial \eta^{1,0},  \, \bar \mu \eta^{1,0}=0 \}, \\
    Y^{1,1}_2 &= Y^{1,1}_3 = \{ \eta^{1,1} = \bar \partial \eta^{1,0} + \partial \eta^{0,1}, \, \bar \mu \eta^{1,0} + \bar \partial \eta^{0,1}=0 \}.
\end{align*}

Note that the spectral sequence at bidegree $(1,1)$ degenerates at most at stage $3$ independently of the dimension of the manifold. In fact $\abs{d_r} = (r, -r+1)$, then
\[
0 \xrightarrow{\quad d_r \quad } E^{1,1}_r \xrightarrow{\quad d_r \quad} 0,
\]
if $r \ge 3$, and $E^{1,1}_\infty = E^{1,1}_3$. If $m=2$, then it degenerates at the $2$-nd stage. Indeed, if $M$ is compact, $J$ is integrable and $m=2$ (i.e., if $M$ is a compact complex surface) or if $M$ is a compact K\"ahler manifold of any dimension, it degenerates at the $1$-st stage and the Dolbeault cohomology group at bidegree $(1,1)$ is isomorphic to the complexified of $H^+$ (cf. \cite{li:zhang}, Theorem 2.16). 
\newline
If $J$ is integrable, we simply have
\begin{align*}
    X^{1,1}_1 &= \{ \alpha^{1,1}: 0= \bar \partial \alpha^{1,1} \}, \\
    X^{1,1}_2 &= \{ \alpha^{1,1}: 0= \bar \partial \alpha^{1,1} = \partial \alpha^{1,1} + \bar \partial \alpha^{2,0} \}, \\
    X^{1,1}_3 &= \{ \alpha^{1,1}: 0= \bar \partial \alpha^{1,1} = \partial \alpha^{1,1} + \bar \partial \alpha^{2,0} = \partial \alpha^{2,0} \}, \\
    Y^{1,1}_1 &= \{ \eta^{1,1} = \bar \partial \eta^{1,0} \}, \\
    Y^{1,1}_2 &= Y^{1,1}_3 = \{ \eta^{1,1} = \bar \partial \eta^{1,0} + \partial \eta^{0,1}, \, \bar \partial \eta^{0,1}=0 \}.
\end{align*}

Before stating the theorem, we make a consideration on the spectral sequence at bidegree $(0,1)$. We have 
\[
0 \xrightarrow{\quad d_3 \quad } E^{0,1}_3 \xrightarrow{\quad d_3 \quad} 0,
\]
so that $E^{0,1}_3 = E^{0,1}_r$ for all $r \ge 3$, but  $E^{0,1}_1 = E^{0,1}_2$ does not imply $E^{0,1}_2 = E^{0,1}_3$, since in general
\[
E^{0,1}_2 \xrightarrow{\quad d_2 \quad} E^{2,0}_2
\]
does not vanish. Consider now the condition $E^{0,1}_1 = E^{0,1}_2$, that will be used in the following theorem. The condition is the same of $X^{0,1}_1/Y^{0,1}_1= X^{0,1}_2/Y^{0,1}_2$. This is particular of bidegree $(0,1)$, and is a consequence of the fact that $Y^{0,1}_1=Y^{0,1}_2$, and that $\partial =0$ on $H^{0,1}_{\bar \partial}$ is equivalent to $d_1=0$ on $E^{0,1}_1$.

\begin{theorem}\label{theo:inclusion:iso}
Let $(M,J)$ be an almost complex manifold. Denote by $\varphi$ the isomorphism $E^{1,1}_r \cong X^{1,1}_r / Y^{1,1}_r$. Then the following conditions are equivalent
\begin{itemize}
    \item [(i)] $\varphi( H^+_\C ) \subseteq H^{1,1}_{Dol}$;
    \item [(ii)] $E^{0,1}_1 \overset{\varphi}{\cong} E^{0,1}_2$.
\end{itemize}
If (i) (or (ii)) holds, then the inclusion is injective.
\end{theorem}
\begin{proof}
Using the isomorphism $\varphi$, (i) is equivalent to prove that 
\[
\frac{X^{1,1}_3}{Y^{1,1}_3} \subseteq \frac{X^{1,1}_1}{Y^{1,1}_1},
\] 
while (ii) is equivalent to 
\[
\frac{X^{0,1}_1}{Y^{0,1}_1} = \frac{X^{0,1}_2}{Y^{0,1}_2}.
\]
By definition, we have $X^{1,1}_3 \subseteq X^{1,1}_1$, and $Y^{1,1}_1 \subseteq Y^{1,1}_2 = Y^{1,1}_3$. Then (i) holds if and only if $Y^{1,1}_1 = Y^{1,1}_2$. Again by definition, $Y^{1,1}_1$ is always a subset of $Y^{1,1}_2$. For the other inclusion, observe that $Y^{1,1}_2 \subseteq Y^{1,1}_1$ if and only if 
\[
\text{for all } \eta^{1,0}, \, \eta^{0,1} \text{ satisfying } \bar \mu \eta^{1,0} + \bar \partial \eta^{0,1}=0,
\]
there exists $\gamma^{1,0}$ such that $\bar \partial \eta^{1,0} + \partial \eta^{0,1} = \bar \partial \gamma^{1,0}$, with $\bar \mu \gamma^{1,0} =0$, i.e.,
\[
E^{0,1}_1 \xrightarrow{\quad d_1=0 \quad } E^{1,0}_1,
\]
or, equivalently, $X^{0,1}_1=X^{0,1}_2$.
Injectivity follows immediately from 
\[
Y^{1,1}_1 \cap X^{1,1}_3 \subseteq Y^{1,1}_1 = Y^{1,1}_3.
\]
\end{proof}

In Example \ref{not:closed}, we will show how the inclusion is not well defined if the condition $E^{0,1}_1 \cong E^{0,1}_2$ is not satisfied, in the case of a complex manifold. In remark \ref{remark:inclusion:welldef}, we will note that example \ref{example:sol3} shows that the same occurs at level of left-invariant Dolbeault cohomology on an almost complex manifolds (the left-invariant cohomology will be introduced in the following sections).
\vspace{.3cm}

For the inclusion up to isomorphism, the key condition is the isomorphism between two terms of the spectral sequence
\begin{equation}\label{degeneracy}
\tag{$\ast$}
    E^{0,1}_1=E^{0,1}_2,
\end{equation}
thus we find meaningful considering the openness and closedness of \eqref{degeneracy} under small deformations of complex structure. \\
Let $M$ be a compact complex manifold and $\{ J_t \} $ a deformation of complex structures on $M$, with small $t \in \C$.
An easy calculation shows that if we assume \eqref{degeneracy} for $t=0$, the function 
\[
e^{0,1}_2 (t) = \dim E^{0,1}_2 (t),
\]
is upper semicontinuous with respect to $t$. Indeed we have upper semicontinuity of $h^{p,q}(t) = \dim H^{p,q}_{\bar \partial} (t)$. Then
\[
e^{0,1}_2 (t) \le h^{0,1} (t) \le h^{0,1} (0) = e^{0,1}_2 (0).
\]
We ask the following 
\medskip

\textbf{Question:} let $(M, J_0)$ be a compact complex manifold. Is condition \eqref{degeneracy} stable under small deformations of the complex structure $J_0$? 
\medskip

In all the examples for which we performed computations, the stability is satisfied.

On the other side, as a consequence of Example \ref{not:closed}, we have that \eqref{degeneracy} is not a closed condition. More precisely, the example shows the following proposition.

\begin{proposition}\label{not:closed:prop}
There exist complex manifolds $(M,J)$ such that 
\begin{itemize}
\item [(i)] the spectral sequence satisfies
\[
E^{0,1}_1  \neq E^{0,1}_2,
\]

\item [(ii)] there are curves of complex structures $ \{ J_t \}$ satysfing $J_0=J$ and
\[
E^{0,1}_1 (t) = E^{0,1}_2 (t)
\]
for all small $t \neq 0$.
\end{itemize}

\end{proposition}

The same is true at left-invariant level for almost complex manifold as shown in Example \ref{deformations:sol3}.

\section{Dolbeault cohomology of Lie Algebras}

Let $\mathfrak{g}$ be a real Lie Algebra of dimension $2m$ and $J$ a complex structure on the vector space $\g$. We call $J$ an \emph{almost complex structure} on the Lie Algebra $\g$. Consider the Chevalley-Eilenberg complex of $\g$, $(A^*_\g, d)$. Recall that the differential is defined as the dual of the Lie bracket $[ \cdot, \cdot]$ for $1$-forms, and extended as a derivation to all forms. $J$ induces a bidegree on the complexified of the Chevalley-Eilenberg complex,
\begin{equation}
    A^k_{\g^\C} = \bigoplus_{p+q=k} A^{p,q}_{\g^\C}.
\end{equation}
The \emph{Dolbeault cohomology of the Lie Algebra $\g$} is the $\bar \partial$-cohomology of the $\bar \mu$-cohomology groups,
\begin{equation}
H^{p,q}_{Dol} (\g) = H^q ( H^{p,*}_{\bar \mu} (\g), \bar \partial),
\end{equation}
and the spectral sequence $\{ E^{*,*}_r (\g) \}_{r \in \N}$ associated to the Hodge filtration
\begin{equation}
    F^p A^k_{\g^\C} =   A^{p,q}_{\g^\C} \cap \ker \bar \mu \oplus \bigoplus_{j \ge p+1} A^{j, k-j}_{\g^\C} .
\end{equation}
is the \emph{spectral sequence of $\g$}.
In the setting of Lie Algebras, it's possible to compute easily the cohomology as a matter of linear algebra and all the spaces are finite dimensional. We set 
\begin{equation}
    b^k_\g = \dim_\C H^k_{dR} (\g, \C),
\end{equation}
and
\begin{equation}
    h^{p,q}_\g = \dim_\C H^{p,q}_{Dol} (\g).
\end{equation}
As a consequence of existence of the spectral sequence, we have Fr\"olicher inequalities (cf. \cite{fro}) for the almost complex case.

\begin{theorem}[\cite{CW:dolbeaultcohomology} Proposition 5.1]\label{fro:ineq:lie}
Let $\mathfrak{g}$ be a real Lie Algebra, $\dim \mathfrak{g} = 2m $ and $J$ an almost complex structure on $\mathfrak{g}$. Then
\begin{equation}
    b^k_g \le \sum \limits_{p+q=k} h^{p,q}_\g.
\end{equation}
Denote by $\chi (\g) = \sum\limits_k (-1)^k b^k$ the Euler characteristic of $\g$. Then    
\begin{equation}
    \chi (\g) = \sum \limits_{p,q} (-1)^{p+q}\,  h^{p,q}_\g.
\end{equation}
\end{theorem}

Consider now a $J$-compatible inner product $\langle \cdot, \cdot \rangle $ on $\g$. It is possible to develop an harmonic theory for differential operators that makes easier some computations of cohomology groups. The \emph{Hodge $*$ operator} is defined as usual by the relation
\begin{equation}
    \langle \varphi, \eta \rangle Vol = \varphi \wedge * \bar \eta,
\end{equation}
with $Vol$ denoting the \emph{volume form} in $A^{2m}_{\g^\C}$, and $\varphi$, $\eta \in A^{p,q}_{\g^\C}$. Taken $\delta$ among $d$, $\mu$, $\partial$, $\bar \partial$, $\bar \mu$, the \emph{formal adjoint of $\delta$} is the operator
\begin{equation}
    \delta^* = - * \bar \delta *.
\end{equation}
The \emph{$\delta$-Laplacian} is defined as
\begin{equation}
    \Delta_\delta = \delta \delta^* + \delta^* \delta,
\end{equation}
and the space of \emph{$\delta$-harmonic $(p,q)$-forms} is
\begin{equation}
\mathcal{H}^{p,q}_\delta = A^{p,q}_{\g^\C} \cap \ker \Delta_\delta.
\end{equation}
On a Lie Algebra, the above spaces are always finite dimensional.
\newline
The operator $\bar \mu^*$ is the adjoint of $\bar \mu$ with respect to $\langle \cdot , \cdot \rangle$, and $A^{p,q}_{\g^\C}$ admits a Hodge decomposition
\begin{equation}
A^{p,q}_{\g^\C} = \mathcal{H}^{p,q}_{\bar \mu} \oplus \bar \mu (A^{p+1,q-2}_{\g^\C}) \oplus \bar \mu^* (A^{p-1,q+2}_{\g^\C}).
\end{equation}
In particular, cohomology classes in $H^{p,q}_{\bar \mu} (\g)$ admit a $\bar \mu$-harmonic representative. For Lie Algebras, the Dolbeault cohomology can also be obtained as the cohomology of the operator $\bar \partial_{\bar \mu}$, defined on $\bar \mu$-harmonic $(p,q)$-forms as 
\begin{equation}
    \bar \partial_{\bar \mu} (\varphi) = \mathcal{H}_{\bar \mu} ( \bar \partial \varphi),
\end{equation}
where we denoted with $\mathcal{H}_{\bar \mu}$ the projection on $\bar \mu$-harmonic forms. It can be checked that $\bar \partial_{\bar \mu}$ is a cohomological operator and
\begin{equation}
   H^{p,q}_{Dol} (\g) \cong \frac{\ker (\bar \partial_{\bar \mu} \colon \mathcal{H}^{p,q}_{\bar \mu} \to \mathcal{H}^{p,q+1}_{\bar \mu})}{\Ima (\bar \partial_{\bar \mu}: \mathcal{H}^{p,q-1}_{\bar \mu} \to \mathcal{H}^{p,q}_{\bar \mu})}.
\end{equation}
Taking the adjoint $\bar \partial_{\bar \mu}^*:= \mathcal{H}_{\bar \mu} \circ \bar \partial^*$, we can consider the associated Laplacian 
\begin{equation}
    \Delta_{\bar \partial_{\bar \mu}}= \bar \partial_{\bar \mu}^* \bar \partial_{\bar \mu} + \bar \partial_{\bar \mu} \bar \partial_{\bar \mu}^*.
\end{equation}
The space of $\bar \partial_{\bar \mu}$-harmonic forms is
\begin{equation}
    \mathcal{H}^{p,q}_{\bar \partial_{\bar \mu}}  = \mathcal{H}^{p,q}_{\bar \mu} \cap \ker \Delta_{\bar \partial_{\bar \mu}}.
\end{equation}
The main obstruction to using $\bar \partial_{\bar \mu}$ to study the Dolbeault cohomology, lies in the fact that in general, $\bar \partial_{\bar \mu}^*$ is not the metric adjoint of $\bar \partial_{\bar \mu}$. However, this is the case in some favorable situation, in particular for compact Lie Groups or for the left-invariant cohomology of solvmanifolds.

\section{Dolbeault cohomology of solvmanifolds endowed with a left-invariant almost complex structure}\label{solvmanifolds}

In this section we extend results obtained from J. Cirici and S. O. Wilson in \cite{CW:dolbeaultcohomology} for nilmanifolds, to the case of solvmanifolds, showing that the left-invariant Dolbeault cohomology always satisfies Serre duality and is described by $\bar \partial_{\bar \mu}$.
\vspace{.5cm}
\newline
Let $M= \Gamma \backslash G$ be a solvmanifold. There is a natural left action $G \times M \longrightarrow M$. Consider the three graded algebra: 
\begin{itemize}
\item $A^*_\mathfrak{g}$, algebra of differential forms on $\mathfrak{g}$;
\item $^L A^*(M)$, algebra of left-invariant forms on $M$;
\item $A^*(M)$, algebra of differential forms on $M$.
\end{itemize}
There is always an isomorphism
\begin{equation}\label{iso:g:gamma:forms}
A^*_\mathfrak{g} \cong \, ^L A^*(M),
\end{equation}
and an inclusion
\begin{equation}
    ^L A^*(M) \xhookrightarrow{\quad} A^*(M).
\end{equation}
Both clearly extend to the complexified algebras. The isomorphism of the Chevalley-Eilenberg complex with left-invariant forms, induces an isomorphism of the de Rham cohomology of $\g$ and the left-invariant de Rham cohomology of $M$ (i.e., the cohomology of left-invariant forms), 
\begin{equation}
    H^*_{dR} (\g; \C) \cong \, ^L H^*_{dR} (M; \C),
\end{equation}
but in general the inclusion is not a quasi-isomorphism, and the left-invariant de Rham cohomology is not isomorphic to the de Rham cohomology of $M$. However, that happens in some favorable case. In particular, if $M$ itself is a Lie Group, or if $G$ is completely solvable, then
\begin{equation}\label{iso:invariant}
H^*_{dR} ( \mathfrak{g}; \C) \cong H^*_{dR} (M; \C).
\end{equation}
\vspace{.3cm}
\newline
Let $J$ be an almost complex structure on the Lie Algebra $\g$. Then via the isomorphism (\ref{iso:g:gamma:forms}), the bigrading of $A^{*,*}_{\g^\C}$ induces a bigrading on the left-invariant complex forms, that corresponds to a left-invariant almost complex structure (still denoted by $J$), and vice-versa, every left-invariant almost complex structure induces an almost complex structure on $J$. $J$ is integrable on $\g$ if and only if it is integrable as an almost complex structure on left-invariant forms. We extend $J$ to non-invariant vector fields by linearity over functions. The extended $\tilde{J}$ is integrable iff its Nijenhuis tensor $N_{\tilde{J}}$ vanishes iff $N_J$ vanishes iff $J$ is integrable on $\g$.
\vspace{.3cm}
\newline
The \emph{left-invariant Dolbeault cohomology of $M$} is defined as the Dolbeault cohomology of the complexified Lie Algebra,
\begin{equation}
^L H^{p,q}_{Dol} (M) = H^{p,q}_{Dol} (\g),
\end{equation}
and the left-invariant spectral sequence of $M$ as the spectral sequence associated to the Dolbeault cohomology of $\g$,
\begin{equation}
    ^L E^{*,*}_r = E^{*,*}_r (\g).
\end{equation}
$M$ also admits a Dolbeault cohomology as an almost complex manifold $(M, \tilde{J})$. It is not known if this cohomology coincides with the left-invariant one, even in the case when (\ref{iso:invariant}) holds, but this is conjectured to be true for nilmanifolds and integrable almost complex structures (cf. \cite{Rollenske}). 

We want to prove the following theorem.
\begin{theorem}\label{main:solvmanifold}
Let $M=\Gamma \backslash G$ be a solvmanifold. Then for all $(p,q)$, its left-invariant Dolbeault cohomology is obtained as $\bar \partial_{\bar \mu}$-harmonic forms,
\begin{equation}\label{harmonic:invariant}
^L H^{p,q}_{Dol}(M) \cong \, ^L \mathcal{H}^{p,q}_{\bar \partial_{\bar \mu}}.    
\end{equation}
The left-invariant spectral sequence satisfies Serre duality at every stage
\begin{equation}\label{invariant:serre}
    ^L E^{p,q}_r \cong \, ^L E^{m-p,m-q}_r, \qquad \forall r \ge 1.
\end{equation}
\end{theorem}

Before giving the proof of the theorem, we make some preliminary observation and state some useful Lemma. \\
The study of left-invariant cohomology is made easy if $\bar \partial_{\bar \mu}^*$ is the metric adjoint of $\bar \partial_{\bar \mu}$. A sufficient condition for this to happen (cf. \cite{CW:dolbeaultcohomology} Lemma 5.2) is 
\begin{equation}\label{suff:cond:1}
\partial \equiv 0 \qquad \text{on $A^{m-1,m}_{\g^\C}$}.
\end{equation}
Equivalent conditions to (\ref{suff:cond:1}) are
\begin{equation}\label{suff:cond:2}
d \equiv 0 \text{ on $A^{2m-1}_{\g^\C}$},
\end{equation}
and
\begin{equation}\label{suff:cond:3}
    H^{2m}_{dR} ( \mathfrak{g}; \mathbb{C} ) \cong \mathbb{C}.
\end{equation}
We recall now some consequences of \eqref{suff:cond:2}, that will be used to prove Theorem \ref{main:solvmanifold}. The fact that $\bar \partial_{\bar \mu}^*$ is adjoint of $\bar \partial_{\bar \mu}$, allows to use harmonic theory to establish an isomorphism from $\bar \partial_{\bar \mu}$-harmonic forms to the Dolbeault cohomology of $\g$, (and consequently to the left-invariant cohomology of $M$).

\begin{proposition}[\cite{CW:dolbeaultcohomology} Theorem 5.4]\label{lie:dolbeault}
Let $\mathfrak{g}$ be a Lie Algebra and $H^{2m}_{dR} (\mathfrak{g}; \C) \cong \mathbb{C}$. Then
\begin{equation}
    \mathcal{H}^{p,q}_{\bar \mu} = \mathcal{H}^{p,q}_{\bar \partial_{\bar \mu}} \oplus \bar \partial_{\bar \mu} (\mathcal{H}^{p,q-1}_{\bar \mu} ) \oplus \bar \partial_{\bar \mu}^* (\mathcal{H}^{p,q+1}_{\bar \mu}),
\end{equation}
and $\bar \partial_{\bar \mu}$-harmonic forms are isomorphic to the Dolbeault cohomology of $\g$,
\begin{equation}
    H^{p,q}_{Dol} (\g) \cong \mathcal{H}^{p,q}_{\bar \partial_{\bar \mu}}.
\end{equation}
\end{proposition}

The Hodge $*$ operator and conjugation give, with the usual argumentation, Serre duality for $\mathcal{H}_{\bar \partial_{\bar \mu}}^{p,q}$, and via the above isomorphism, the first stage of the left-invariant spectral sequence also satisfies Serre duality.

\begin{proposition}[\cite{CW:dolbeaultcohomology} Corollary 5.5]\label{lie:serre}
Let $\mathfrak{g}$ be a Lie Algebra and $H^{2m}(\mathfrak{g}; \mathbb{C}) \cong \mathbb{C}$. Then
\[
H^{p,q}_{Dol} (\g) \cong H^{m-p,m-q}_{Dol} (\g).
\]
\end{proposition}

For the sake of completeness, we recall the proof of the following well known result.
\begin{lemma}\label{lemma:unimodular}
Let $\mathfrak{g}$ be an unimodular Lie Algebra, $\dim_{\mathbb{R}} \mathfrak{g}= 2m$. Then $d \equiv 0$ on $A^{2m-1}_{\g^\C}$.
\end{lemma}
\begin{proof}
Let $\{ e_j \}_{j=1}^{2m}$ be a basis of $\mathfrak{g}$, and $\{ \phi^j \}_{j=1}^{2m}$ a dual basis. Set
\[
[e_j, e_k] = \sum_l C_{jk}^l e_l, \qquad C_{jk}^l + C_{kj}^l=0
\]

The differential on $1$-forms is determined by the structure constants
\[
d \phi^l = - \frac{1}{2} \sum_{j,k} C_{jk}^l \phi^j \wedge \phi^k.
\]
A basis of $(2m-1)$-forms is given by $\{ \phi^1 \wedge \dots \wedge \hat \phi^j \wedge \dots \wedge \phi^{2m} \}_{j=1}^{2m}$, where $\hat \phi^j$ means that the form is omitted. Then we have
\begin{multline*}
d(\phi^1 \wedge \dots \wedge \hat \phi^j \wedge \dots \wedge \phi^{2m} ) = \sum_{k < j} (-1)^k d\phi^k \wedge \phi^1 \wedge \dots \wedge \hat \phi^k \wedge \dots \wedge \hat \phi^j \wedge \dots \wedge \phi^{2m} + \\[.5ex]
+ \sum_{k > j} (-1)^k d\phi^k \wedge \phi^1 \wedge \dots \wedge \hat \phi^j \wedge \dots \wedge \hat \phi^k \wedge \dots \wedge \phi^{2m} = \\[.5ex]
 =- \sum_{k < j} (-1)^k \Big(  \sum_{l<n} C_{ln}^k \phi^l \wedge \phi^n \Big) \wedge \phi^1 \wedge \dots \wedge \hat \phi^k \wedge \dots \wedge \hat \phi^j \wedge \dots \wedge \phi^{2m} + \\[.5ex]
 - \sum_{k > j} (-1)^k \Big(  \sum_{l<n} C_{ln}^k \phi^l \wedge \phi^n \Big) \wedge \phi^1 \wedge \dots \wedge \hat \phi^j \wedge \dots \wedge \hat \phi^k \wedge \dots \wedge \phi^{2m} = \\[.5ex]
 =- \sum_{k < j} (-1)^k C_{kj}^k \phi^k \wedge \phi^j \wedge \phi^1 \wedge \dots \wedge \hat \phi^k \wedge \dots \wedge \hat \phi^j \wedge \dots \wedge \phi^{2m} + \\[.5ex]
 - \sum_{k > j} (-1)^k C_{jk}^k \phi^j \wedge \phi^k \wedge \phi^1 \wedge \dots \wedge \hat \phi^j \wedge \dots \wedge \hat \phi^k \wedge \dots \wedge \phi^{2m} = \\[.5ex]
= (-1)^j \sum_{k} C_{jk}^k Vol = (-1)^j \, Tr(ad_{e_j} ) Vol,
\end{multline*}
where $Vol = \phi^1 \wedge \dots \wedge \phi^{2m}$, and the last equality follows by definition of trace
\[
Tr( ad_{e_j} ) = \sum_k \, \langle ad_{e_j} (e_k), e_k \rangle = \sum_k C_{jk}^k. 
\]
If $\mathfrak{g}$ is unimodular, the trace of the adjoint vanishes.
\end{proof}

We are ready to proceed with the proof.

\begin{proof}[Proof of Theorem \ref{main:solvmanifold}.]
Since $G$ is a connected, simply connected solvable Lie Group that admits a lattice, it is unimodular and by Lemma \ref{lemma:unimodular}, condition \eqref{suff:cond:2} is satisfied. 
\newline
Proposition \ref{lie:dolbeault} and \ref{lie:serre} are valid for $\g$ and so the left-invariant Dolbeault cohomology group of $M$ are isomorphic to left-invariant $\bar \partial_{\bar \mu}$-harmonic forms and satisfy Serre duality. This proves \eqref{harmonic:invariant}, and also \eqref{invariant:serre} for the first stage.
\newline
For $r=1$, note that $(^L E^{*,*}_1, d_1)$ satisfies the hypothesis of the main theorem in \cite{Mil:serre}. In fact we proved that $d$ vanishes on $A^{2m-1}_{\g^\C}$, thus also $\partial$, $\bar \partial$ and $\bar \mu$ vanish on $A^{2m-1}_{\g^\C}$. $d_1$ is a sum of such differentials, thus $d_1=0$. Serre duality at first stage, $d_1=0$ on $A^{2m-1}_{\g^\C}$ and $^L E^{m,m}_1 \cong \C$ imply Serre duality at every stage.
\end{proof}

\section{Examples}

In this section we show some example of what we proved in section \ref{contained} and in section \ref{solvmanifolds}.
\newline
We begin showing how, on a complex manifold, that \eqref{degeneracy} is not a closed condition. We also note that when it is not satisfied, cohomology classes in $H^+_\C$ do not define cohomology classes in $H^{1,1}_{\bar \partial}$.

\begin{example}[$E^{0,1}_1 = E^{0,1}_2$ is not a closed condition]\label{not:closed}
We provide two examples of not closedness of condition 
\begin{equation}
    \tag{$\ast$}
    E^{0,1}_1 = E^{0,1}_2.
\end{equation}
The first one is provided in \cite{ceballos:ugarte}, Example 4.8. We recall here the construction. Consider the $6$-dimensional real Lie Algebra spanned by $\{ e_j \}$, $j =1, \dots, 6$. The differentials of the dual basis are
\[
de^1 = 0, \quad de^2=0, \quad de^3 =0, \quad de^4 = e^{12}, \quad de^5 = e^{13} + e^{42}, \quad de^6 = e^{14} + e^{23}.
\]
Consider the family of complex structure parametrized by $t$, given by
\begin{align*}
J_t e^1 &= -\sqrt{ \frac{3 (3 - \sin t)(7 + 3 \sin t)}{(5 + \sin t)(11- \sin t)} } e^2, \\
J_t e^3 &= \sqrt{ \frac{3(3-\sin t)(11- \sin t)}{(5+\sin t)(7+3 \sin t)} } e^4, \\
J_t e^5 &= - \sqrt{ \frac{(11 - \sin t)(7 + 3 \sin t)}{3(3 - \sin t)(5 + \sin t)}} e^6,
\end{align*}
and the nilpotent complex Lie Algebra obtained as 
\begin{align*}
4 \varphi_t^1 &= \sqrt{(11 - \sin t)(5 + \sin t)} e^1 + i \sqrt{3(3 - \sin t)(7 + 3 \sin t)} e^2 \\
8 \varphi^2_t &= (5 + \sin t)(7 + 3 \sin t) e^3 - i
\sqrt{ 3(5 + \sin t)(3 - \sin t)(11 - \sin t)(7 + 3 \sin t)} e^4, \\
\end{align*}
and
\begin{multline*}
    128 \varphi^3_t = (5 + \sin t)(7 + 3 \sin t)
\Big[ 3(3 - \sin t) \sqrt{ (11 - \sin t)(5 + \sin t)} e^5 \\ + i(11 - \sin t)
\sqrt{
3(3 - \sin t)(7 + 3 \sin t)} e^6 \Big].
\end{multline*}
The  basis of $(1,0)$-forms parametrized by $t$, has the following differentials:
\[
d \varphi^1_t =0, \quad d \varphi^2_t = \varphi^{1 \bar 1}_t, \quad d \varphi^3_t = \frac{1-\sin t}{2} \varphi^{12}_t + 2 \varphi^{1 \bar 2}_t + \frac{1 + \sin t}{4} \varphi^{2 \bar 1}_t.
\]
If $\abs{ \sin t} \neq 1$, we have that the first stage of the spectral sequence coincides with the second one (but not with the third one). In particular, condition \eqref{degeneracy} is satisfied. For $\sin t =1$, we  have
\[
d \varphi^1_t = 0, \quad d \varphi^2_t = \varphi^{1 \bar 1}_t, \quad d \varphi^3_t = \varphi^{1 \bar 2}_t + \frac{1}{2}\varphi^{2 \bar 1}_t,
\]
thus $H^{0,1}_{\bar \partial} \cong \C^3$, while $E^{0,1}_2 \cong \C^2$ and \eqref{degeneracy} is not satisfied. \\
Note that $d ( \varphi^{\bar 3}_t + \varphi^{\bar 3}_t)  = \frac{1}{2} (\varphi^{1 \bar  2}_t + \varphi^{ 2 \bar  1}_t ) = \partial \varphi^{\bar 3}_t + \bar \partial \varphi^3_t$ is a $d$-exact real $2$-form and thus belongs to the $0$ class in $H^+_\C$. However, it does not belong to the $0$ class in the Dolbeault cohomology of $(1,1)$-forms since is not the $\bar \partial$ of any $(1,0)$-form, and the inclusions $H^+ \subseteq H^{1,1}_{\bar \partial}$ and so $H^+_\C \subseteq H^{1,1}_{\bar \partial} $ are not well defined.
\vspace{.3cm}

The second example comes from the holomorphically parallelizable Nakamura manifold, and its small deformations. The manifold can be obtained as the quotient of the Lie Group $G= \C \ltimes_\phi \C^2$, with
\[
\phi(z) = \begin{bmatrix}
e^z & 0 \\
0 & e^{-z} \\
\end{bmatrix},
\]
by the lattice $\Gamma = \langle a + ib, c+id \rangle$. If we assume that $b, d \in \pi \Z$ (cf. \cite{angella}, Example 3.4), we can compute the Dolbeault cohomology of the manifold using the $(0,1)$-forms
\[
d z^{\bar 1}, \quad e^{z_1} d z^{\bar 2}, \quad e^{z_1} d z^{\bar 3},
\]
and the conjugate $(1,0)$-forms.  The deformation we are interested in, seen as an element of $T^{0,1}M^* \otimes T^{1,0}M$, is 
\[
t \, d z^{\bar 1} \otimes \frac{\partial}{\partial z_1}.
\]
Direct calculations show that for all $t \neq 0$, the spectral sequence degenerates at the first stage. For $t=0$, we have $H^{0,1}_{\bar \partial} \cong \C^3$ and $E^{0,1}_2 \cong \C$, thus \eqref{degeneracy} is not satisfied. This fact is resumed in \cite{angella:kasuya:deformations}, Proposition 4.2. \\
The above manifolds, both provide a proof of Proposition \ref{not:closed:prop}.
\end{example}

\vspace{.3cm}

We proceed now to compute the Dolbeault left-invariant cohomology of $4$-dimensional solvmanifolds that do not admit integrable almost complex structures. There exists three such solvmanifolds (cf. \cite{Bock}). The first one is a nilmanifold. The second one is a symplectic, completely solvable but not nilpotent solvmanifold (Example \ref{example:sol3}) and the last one is a completely solvable but not symplectic nor nilpotent solvmanifold (Example \ref{example:G}). Calculations for the nilmanifold have already been made in \cite{CW:dolbeaultcohomology}, Example 5.15.

\begin{example}[$\Gamma \backslash Sol(3) \times \mathbb{S}^1$]\label{example:sol3}
Denote by $Sol(3)$ the solvable Lie Group of dimension $3$. It can be obtained considering the groups $(\mathbb{R}, +)$, $(\mathbb{R}^2, +)$ and taking the semi-direct product $Sol(3)= \mathbb{R} \ltimes_\phi \mathbb{R}^2$, with
\[
\phi(t) = 
\begin{bmatrix}
e^t & 0 \\
0 & e^{-t} \\
\end{bmatrix}.
\]
The product $Sol(3) \times \mathbb{R}$ can be identified as a subgroup of matrices via the homomorphism
\[
(t, x, y , s) \xmapsto{\quad \theta \quad}
\begin{bmatrix}
e^t & 0 & x & 0 & 0 \\
 & e^{-t} & y & 0 & 0 \\
 &  & 1 & 0 & 0 \\
 &  &   & 1 & s \\
 &  &   &   & 1 \\
\end{bmatrix}
 \in SL(5, \mathbb{R}).
\]
Denote with $K$ the image of $Sol(3) \times \mathbb{R}$ by $\theta$. Then $K$ is a subgroup of $SL(5, \mathbb{R})$ with respect to matrix multiplication, isomorphic to $Sol(3) \times \mathbb{R}$. By \cite{auslander} (Theorem 4), $K$ admits a lattice $\Gamma$. The quotient $M = \Gamma \backslash K$ is a solvmanifold, of real dimension $4$. An explicit construction of $M$ can be found in \cite{Bock} in the examples following the classification of four-dimensional solvmanifolds. Taking $A \in K$ and computing $A^{-1} dA$, we obtain a basis of left-invariant forms
\[
\Big \{ e^1 = dt, e^2 = e^{-t} dx, e^3 = e^t dy, e^4 = ds \Big \}.
\]
The dual basis of vector fields is
\[
\Big \{ e_1 = \frac{\partial}{\partial t},e_2 = e^t \frac{\partial}{\partial x},e_3 = e^{-t} \frac{\partial}{\partial y}, e_4 = \frac{\partial}{\partial s} \Big \}.
\]
The only non-zero brackets are 
\[
[ e_1, e_2 ] = e_2, \qquad [e_1, e_3 ] = -e_3, 
\]
and the differentials of $1$-forms can be obtained directly differentiating $\{ e^j \}$, or by taking the dual of the Lie bracket. Differentials of all forms are obtained by Leibniz rule. We have
\begin{align*}
\text{ $1$-forms: } \quad & de^1 =0, \quad de^2 = - e^{12}, \quad de^3 = e^{13}, \quad de^4 = 0; \\[1ex]
\text{ $2$-forms: } \quad & d e^{12} =0, \quad d e^{13} =0, \quad de^{14} =0, \quad de^{23}=0, \\[1ex]
& de^{24} = -e^{124}, \quad de^{34} = e^{134}; \\[1ex]
\text{$3$-forms: } \quad & de^{123}=0, \quad de^{124} =0, \quad de^{134} =0, \quad de^{234}=0. \\[1ex]
\end{align*}
The differential vanishes on $4$-forms for degree reasons. We can directly compute the left-invariant de Rham cohomology of $M$ and, since $Sol(3) \times \R$ is completely solvable, it coincides with the real de Rham Cohomology. The kernels of the differential $d$ are
\[
Z^1_d = \langle e^1, e^4 \rangle, \quad Z^2_d = \langle e^{12}, e^{13}, e^{14}, e^{23} \rangle, \quad Z^3_d = A^3 (M; \mathbb{R}); 
\]
and the images are 
\[
B^1_d = \{ 0 \}, \quad B^2_d = \langle e^{12}, e^{13} \rangle, \quad B^3_d =\langle e^{124}, e^{134} \rangle,
\]
so that the de Rham cohomology of $M$ is 
\begin{align*}
    H^0_{dR} (M; \mathbb{R}) &= \langle 1 \rangle, \\
    H^1_{dR} (M; \mathbb{R}) &= \langle e^1, e^4 \rangle, \\
    H^2_{dR} (M; \mathbb{R}) &= \langle e^{14}, e^{23} \rangle, \\
    H^3_{dR} (M; \mathbb{R}) &= \langle e^{123}, e^{234} \rangle, \\
    H^4_{dR} (M; \mathbb{R}) &= \langle e^{1234} \rangle. \\
\end{align*}
Betti numbers of $M$ are listed in Table \ref{sol3:betti}.
\begin{table}[H]
\centering
\[
\begin{array}{c|c|c|c|c|c}
       n & 0 & 1 & 2 & 3 & 4 \\
       \hline
       b^n &  1 & 2 & 2 & 2 & 1 \\
\end{array}
\]
\caption{Betti numbers for $\Gamma \backslash Sol(3)\times \mathbb{S}^1$.}
\label{sol3:betti}
\end{table}
In particular, the Euler characteristic of $M$ is $0$. The complex de Rham cohomology of $M$ is obtained as the complexified of the real one. 
\newline 
The bigrading induced on it from the spectral sequence will be clear by the full computation of the last stage of the sequence, and in general it's dependent on the almost complex structure chosen on $M$. We shall consider three almost complex structures on $M$. Set 
\[
\text{(A)}\quad  
\begin{cases}
Je_1 = e_2, \\
Je_3 = e_4, \\
\end{cases}
\qquad 
\text{(B)}\quad  
\begin{cases}
Je_1 = e_3, \\
Je_2 = e_4, \\
\end{cases}
\qquad 
\text{(C)}\quad  
\begin{cases}
Je_1 = e_4, \\
Je_2 = e_3. \\
\end{cases}
\]
The complex structure (A) induces a decomposition of the complexified tangent space. A basis is obtained taking projections of $e_1$ and $e_3$. Define
\[
Z_1 = \pi^{1,0} (e_1)= \frac{1}{2} ( e_1 - i e_2 ), \qquad Z_2 = \pi^{1,0}(e_3)= \frac{1}{2} (e_3 -ie_4).
\]
A basis of the complexified tangent space is $\{ Z_1, Z_2, \bar Z_1, \bar Z_2 \}$, and a basis of dual forms for the complexified Lie Algebra is $ \{ \varphi^1, \varphi^2, \bar \varphi^1, \bar \varphi^2 \}$, where
\[
\varphi^1 = e^1 +ie^2, \qquad \varphi^2 = e^3 +i e^4.
\]
As for the de Rham cohomology, the differentials $\mu$, $\partial$, $\bar \partial$, $\bar \mu$ of complex forms can be obtained calculating the differential of the complex forms starting from the real ones, then separating the bidegrees, or else by duality from the brackets. We shall use here the second way. The brackets of complex vector fields are
\[
[Z_1, Z_2] = \frac{1}{4} [e_1 - i e_2, e_3 - i e_4] =-\frac{1}{4}e_3 = -\frac{1}{2} Re \, Z_2 = -\frac{1}{4}( Z_2 + \bar Z_2).
\]
In the same way, we obtain
\[
[Z_1, Z_2] = [ Z_1, \bar Z_2] = [ \bar Z_1, Z_2] = [\bar Z_1, \bar Z_2] = -\frac{1}{4}( Z_2 + \bar Z_2),
\]
\[ [Z_1, \bar Z_1] = -\frac{1}{2}( Z_1 - \bar Z_1).
\]
The differential $\bar \mu \varphi^1$ is a $(0,2)$ form. It will be determined by the bracket on $(0,1)$ vector fields and duality relations
\[
\bar \mu \varphi^1 (\bar Z_1, \bar Z_2 ) = - \varphi^1 ( [\bar Z_1, \bar Z_2] ) = 0,
\]
i.e., $\bar \mu \varphi^1 =0$. In the same way, $\bar \partial \varphi^1$ is of type $(1,1)$, and we have
\begin{align*}
\bar \partial \varphi^1 ( Z_1, \bar Z_1) = -\varphi^1 ( [Z_1, \bar Z_1] ) =\frac{1}{2}, \\[1ex]
\bar \partial \varphi^1 (Z_1, \bar Z_2)=-\varphi^1 ( [Z_1, \bar Z_2] ) = 0, \\[1ex]
\bar \partial \varphi^1 ( Z_2, \bar Z_1) =-\varphi^1 ( [Z_2, \bar Z_1] ) =0, \\[1ex]
\bar \partial \varphi^1 ( Z_2, \bar Z_2) =-\varphi^1 ( [Z_2, \bar Z_2] ) =0, \\[1ex]
\end{align*}
so that $\bar \partial \varphi^1 = \frac{1}{2} \varphi^{1 \bar 1}$. We obtain:
\begin{alignat*}{5}
&\text{ $1$-forms: } \quad &&\mu \varphi^1 = 0,  \quad &&\partial \varphi^1 =0, \quad &&\bar \partial \varphi^1 = \frac{1}{2} \varphi^{1 \bar 1}, \quad &&\bar \mu \varphi^1 =0, \\[1ex]
& &&\mu \varphi^2 = 0, \quad &&\partial \varphi^2 =\frac{1}{4} \varphi^{12}, \quad &&\bar \partial \varphi^2 = \frac{1}{4} (\varphi^{1 \bar 2} + \varphi^{\bar 1 2 } ), \quad &&\bar \mu \varphi^2 =\frac{1}{4} \varphi^{\bar 1 \bar 2}, \\[1ex]
\end{alignat*}
and the conjugate equations. Note that the almost complex structure is not integrable since $\bar \mu \varphi^2 \neq 0$. From the differentials on $1$-form, we can compute them for all degrees. Every equation has to be conjugated.
\begin{alignat*}{3}
&\text{ $2$-forms: } \quad &&\begin{cases}
\mu \varphi^{12} = \partial \varphi^{12} = 0, \\[2ex]
\bar \partial \varphi^{12} = -\frac{1}{4} \varphi^{12 \bar1}, \\[2ex]
\bar \mu \varphi^{12} = -\frac{1}{4} \varphi^{1 \bar 1 \bar 2},\\[2ex]
\end{cases} \qquad \qquad 
&&\begin{cases}
\mu \varphi^{1\bar 2} = \bar \mu \varphi^{1 \bar 2} = 0, \\[2ex]
\partial \varphi^{1 \bar 2} = \frac{1}{4} \varphi^{12 \bar1}, \\[2ex]
\bar \partial \varphi^{1 \bar 2} = \frac{1}{4} \varphi^{1 \bar 1 \bar 2},\\[2ex]
\end{cases}
\\[2ex]
& &&\begin{cases}
\mu \varphi^{1 \bar 1} = \bar \mu \varphi^{1 \bar 1} = 0, \\[2ex]
\partial \varphi^{1 \bar 1} = \bar \partial \varphi^{1 \bar 1} =0 \\[2ex]
\end{cases} \qquad \qquad 
&&\begin{cases}
\mu \varphi^{2 \bar 2} = \bar \mu \varphi^{2 \bar 2} = 0, \\[2ex]
\partial \varphi^{2 \bar 2} = \frac{1}{2} \varphi^{12 \bar 2}, \\[2ex]
\bar \partial \varphi^{2 \bar 2} = -\frac{1}{2} \varphi^{2 \bar 1 \bar 2},\\[2ex]
\end{cases}
\end{alignat*}
On forms of degree $3$ and $4$, all the differentials are $0$. We are ready to compute the left-invariant spectral sequence. As a consequence of Theorem \ref{main:solvmanifold}, it satisfies Serre duality. We show calculations only for bidegree $(1,1)$. On
\[
A^{1,1}_\C = \langle \varphi^{1 \bar 1}, \varphi ^{1 \bar 2}, \varphi^{\bar 1 2}, \varphi ^{2 \bar 2} \rangle,
\]
$\bar \mu$ vanishes for bidegree reasons. $\bar \partial$ vanishes on $\varphi^{1 \bar 1}$, and $\bar \partial \varphi^{1 \bar 2} = \bar \partial \varphi^{\bar 1 2 } = \bar \mu \varphi^{12}$, while $\bar \partial \varphi^{2 \bar 2}$ does not belong to the image of $\bar \mu$, then
\[
X^{1,1}_1 = \langle \varphi^{1 \bar 1}, \varphi^{1 \bar 2}, \varphi^{\bar 1 2} \rangle.
\]
For bidegree reason, $(1,1)$-forms cannot be in the image of $\bar \mu$, so we just have to check the existence of a $(1,0)$-form $\beta^{1,0}$ such that $\eta^{1,1} = \bar \partial \beta^{1,0}$ and $\bar \mu \beta^{1,0} =0$. In general, $\beta^{1,0} = a \varphi^1 + b \varphi^2$. The condition $\bar \mu \beta^{1,0} =0$ gives
\[
\bar \mu \beta^{1,0} = \frac{b}{4} \varphi^{\bar 1 \bar 2} = 0,
\]
so it has to be $\beta^{1,0} = a \varphi^1$. Then
\[
\bar \partial \beta^{1,0} = \frac{a}{2} \varphi^{1 \bar 1},
\]
so that $Y^{1,1}_1 = \langle \varphi^{1 \bar 1} \rangle$. Taking the quotient, we have
\[
^L H^{1,1}_{Dol} (M) = \langle \varphi^{1 \bar 2 }, \varphi^{ \bar1 2} \rangle \cong \mathbb{C}^2.
\]
Analogous calculations show that 
\[
^L H^{1,0}_{Dol} (M) =\, ^L H^{2,0}_{Dol} (M) =\, ^L H^{0,2}_{Dol} (M) =\, ^L H^{1,2}_{Dol} (M) = \{ 0 \},
\]
\[
^L H^{0,1}_{Dol} (M) = \langle \varphi^{\bar 1}, \varphi^{\bar 2} \rangle \cong \mathbb{C}^2,
\]
\[
^L H^{2,1}_{Dol} (M) = \langle \varphi^{12\bar 1}, \varphi^{12\bar 2} \rangle \cong \mathbb{C}^2,
\]
\[
^L H^{0,0}_{Dol} (M) = \langle 1 \rangle \cong \mathbb{C}, 
\]
\[
^L H^{2,2}_{Dol} (M) = \langle \varphi^{12 \bar 1 \bar 2} \rangle \cong \mathbb{C}.
\]
The first stage is 
\[
^L_A E^{*,*}_1 \cong 
\begin{array}{|c|c|c|} \hline
      \rule{0pt}{1em} 0 & 0 & \mathbb{C} \\ \hline
      \rule{0pt}{1em} \mathbb{C}^2 & \mathbb{C}^2 & \mathbb{C}^2 \\ \hline
      \rule{0pt}{1em} \mathbb{C} & 0 & 0 \\ \hline
\end{array}.
\]
Note that the spectral sequence degenerates at the first stage because separately in every bidegree, the dimension can only decrease and at the last stage their sum must coincide with the Betti numbers of $M$. To compute the real $J$-invariant group $H^+$, we write forms in $H^2_{dR}$ as complex forms:
\[
e^{14} = \frac{\varphi^1 + \varphi^{\bar 1}}{2} \wedge \frac{\varphi^2 - \varphi^{\bar 2}}{2i} = \frac{1}{4i} ( \varphi^{12} + \varphi^{\bar 1 2} - \varphi^{1 \bar 2} - \varphi^{\bar 1 \bar 2} ),
\]
\[
e^{23} = \frac{\varphi^1 - \varphi^{\bar 1}}{2i} \wedge \frac{\varphi^2 + \varphi^{\bar 2}}{2} = \frac{1}{4i} ( \varphi^{12} - \varphi^{\bar 1 2} + \varphi^{1 \bar 2} - \varphi^{\bar 1 \bar 2}).
\]
First adding then subtracting we obtain $e^{14} - e^{23} \in H^+$ and $e^{14} + e^{23} \in H^-$,
\[
H^+ = \langle \frac{1}{2i} ( \varphi^{\bar 12} - \varphi^{1 \bar 2} ) \rangle, \qquad H^- = \langle \frac{1}{2i} ( \varphi^{12} - \varphi^{\bar 1 \bar 2})\rangle.
\]
\newline
Direct calculations for structure (B), lead to the same equations for the differential of $1$-forms, and thus to the same spectral sequence. For structure (C), set
\[
Z_1 = \frac{1}{2} ( e_1 - i e_4), \qquad Z_2 = \frac{1}{2}( e_2 -i e_3),
\]
and denote with $\varphi^1$, $\varphi^2$ the dual forms. The non-zero brackets are
\[
[Z_1, Z_2] = [\bar Z_1, Z_2]  =  \frac{1}{2} \bar Z_2,
\]
\[
[\bar Z_1, \bar Z_2] = [ Z_1, \bar Z_2]  =  \frac{1}{2} Z_2,
\]
and the differentials on $1$-forms are
\begin{alignat*}{4}
&\mu \varphi^1 = 0,  \quad &&\partial \varphi^1 =0, \quad &&\bar \partial \varphi^1 = 0, \quad &&\bar \mu \varphi^1 =0, \\[1ex]
&\mu \varphi^2 = 0, \quad &&\partial \varphi^2 =0, \quad &&\bar \partial \varphi^2 = -\frac{1}{2} \varphi^{1 \bar 2}, \quad &&\bar \mu \varphi^2 =-\frac{1}{2} \varphi^{\bar 1 \bar 2}. \\[1ex]
\end{alignat*}
On $2$-forms, the only non-vanishing differentials are 
\[
\bar \mu \varphi^{12} = \frac{1}{2} \varphi^{1 \bar 1 \bar 2}, \qquad  \partial \varphi^{1 \bar 2} = - \frac{1}{2} \varphi^{12 \bar 1}, 
\]
and the conjugate equations. All the differential vanish on forms of degree $3$ and $4$. The computation of Dolbeault cohomology group is straightforward:
\[
^L H^{0,0}_{Dol} (M) = \langle 1 \rangle, \qquad ^L H^{2,2}_{Dol} (M) = \langle \varphi^{12 \bar 1 \bar 2} \rangle,
\]
\[
^L H^{1,1}_{Dol} (M) = \langle \varphi^{1 \bar 1 }, \varphi^{2 \bar 2} , \varphi^{1 \bar 2}, \varphi^{\bar 1 2} \rangle \cong \mathbb{C}^4,
\]
\[
^L H^{0,1}_{Dol} (M) = \langle \varphi^{\bar 1}, \varphi^{\bar 2} \rangle, \qquad ^L H^{2,1}_{Dol} (M) = \langle \varphi^{1 2 \bar 1}, \varphi^{1 2 \bar 2} \rangle,
\]
\[
^L H^{1,0}_{Dol} (M)= \langle \varphi^1 \rangle, \qquad ^L H^{1,2}_{Dol} (M) = \langle \varphi^{2 \bar 1 \bar 2} \rangle,
\]
and the first stage of the spectral sequence is
\[
^L_C E^{*,*}_1 \cong 
\begin{array}{|c|c|c|} \hline
      \rule{0pt}{1em} 0 & \mathbb{C} & \mathbb{C} \\ \hline
      \rule{0pt}{1em} \mathbb{C}^2 & \mathbb{C}^4 & \mathbb{C}^2 \\ \hline
      \rule{0pt}{1em} \mathbb{C} & \mathbb{C} & 0 \\ \hline
\end{array}
\]
Note that this is not the bigrading induced on the complex de Rham cohomology, since for example $\dim \, ^L H^{1,0}_{Dol} + \dim \, ^L H^{0,1}_{Dol} \gneqq \dim H^1_{dR}$. The same happens for degree $2$ and $3$. The following stage is obtained as $E^{p,q}_2 \cong X^{p,q}_2/Y^{p,q}_2$. The quotients are
\[
^L E^{0,0}_2 = \langle 1 \rangle, \qquad ^L E^{2,2}_2 = \langle \varphi^{12 \bar 1 \bar 2} \rangle,
\]
\[
^L E^{1,0}_2 = \langle \varphi^1 \rangle, \qquad ^L E^{1,2}_2 = \langle \varphi^{2 \bar 1 \bar 2} \rangle,
\]
\[
^L E^{0,1}_2 = \langle \varphi^{\bar 1} \rangle, \qquad ^L E^{2,1}_2 = \langle \varphi^{12 \bar2 } \rangle, 
\]
\[
^L E^{2,0}_2 =\,  ^L E^{0,2}_2  = \{ 0 \},
\]
\[
^L E^{1,1}_2 = \langle \varphi^{1 \bar 1}, \varphi^{2 \bar 2} \rangle.
\]
For dimension reasons, this is also the $\infty$ stage of the spectral sequence,
\[
^L_C E^{*,*}_r \cong 
\begin{array}{|c|c|c|} \hline
      \rule{0pt}{1em} 0 & \mathbb{C} & \mathbb{C} \\ \hline
      \rule{0pt}{1em} \mathbb{C} & \mathbb{C}^2 & \mathbb{C} \\ \hline
      \rule{0pt}{1em} \mathbb{C} & \mathbb{C} & 0 \\ \hline
\end{array},
 \qquad \forall r \ge 2,
\]
and gives the induced bigrading on the de Rham cohomology. Proceeding as for structure (A), 
\[
H^+ = \langle i \varphi^{1 \bar 1}, i \varphi^{2 \bar 2} \rangle, \qquad H^- = \{ 0 \}.
\]
\end{example}

\begin{remark}\label{remark:inclusion:welldef}
In both examples, harmonic representatives of $H^+_\C$ are also harmonic representative of $^L H^{1,1}_{Dol}$. For structure (A), the condition $^L E^{0,1}_1 = \, ^L E^{0,1}_2$ is satisfied, and if we consider a non-harmonic representative in $H^+_\C$, it still defines a class in $^L H^{1,1}_{Dol}$, as expected from Theorem \ref{theo:inclusion:identity}. In fact we have $H^+_\C = \langle \varphi^{\bar 1 2} - \varphi^{1 \bar 2} \rangle$. The form is $d$-closed. $d$-exact $(1,1)$ forms are written as $\bar \partial \beta^{1,0} + \partial \beta^{0,1}$, with
\[
\beta^{1,0} = a \varphi^1 + b \varphi^2, \qquad \beta^{0,1} = c \varphi^{\bar 1} + d \varphi^{\bar 2},
\]
with the conditions $\bar \mu \beta^{1,0} + \bar \partial \beta^{0,1}$ and $\mu \beta^{0,1} + \partial \beta^{1,0}$ that are satisfied only if $b+d=0$. Immediately we have
\[
\bar \partial \beta^{1,0} + \partial \beta^{0,1} = \frac{a-c}{2} \varphi^{1 \bar 1},
\]
that is the $0$ class in $^L H^{1,1}_{Dol}$.
\newline
This is not true for structure (C), in fact if we modify $\varphi^{1 \bar 1}$ with a $d$-exact $(1,1)$-form, the class in $^L H^{1,1}_{Dol}$ varies.
\end{remark}

\begin{example}[Invariant deformations of $\Gamma \backslash Sol(3) \times \mathbb{S}^1$]\label{deformations:sol3}
In this example we compute deformations of the manifold $M= \Gamma \backslash Sol(3) \times \mathbb{S}^1$. Consider the almost complex solvmanifold endowed with a left-invariant almost complex structure $(M,J)$, where $J$ is the structure (C) of Example \ref{example:sol3}. We study the behaviour of the left-invariant spectral sequence under deformations.\\
Since we are interested in the left-invariant cohomology, calculating deformations of $J$ is a matter of linear algebra. As a matrix, $J$ is written as 
\[
J= \begin{bmatrix}
0 & -Id_2 \\
Id_2 & 0 \\
\end{bmatrix},
\]
and its small deformations are represented by a $4 \times 4$ matrix $L$ satisfying
\[
LJ + JL =0. 
\]
This last condition is written as
\[
L = \begin{bmatrix}
A & B \\
PBP & - PAP \\
\end{bmatrix}
\]
where $A$ and $B$ are $2 \times 2$ matrices and $P= \begin{bmatrix}
0 & 1 \\ 1 & 0 \\
\end{bmatrix}$. \\
The deformations are also codified by a form $\psi \in T^{1,0} M \otimes T^{0,1}M^*$ that is written as
\[
\psi = \psi_1^1 \varphi^{\bar 1} \otimes Z_1 + \psi_1^2 \varphi^{\bar 1} \otimes Z_2 + \psi_2^1 \varphi^{\bar 2} \otimes Z_1 + \psi_2^2 \varphi^{\bar 2} \otimes Z_2,
\]
and must satisfy $\psi = \frac{1}{2} (L - i JL)$. By writing out both members of the equality we obtain the expression of $\psi$ in function of $A$ and $B$, then we compute the brackets of the deformed structures in function of $\psi$ and the brackets at time $0$. Finally by duality we obtain the differentials of the deformed left-invariant forms, and compute the left-invariant spectral sequence. We classify deformations into two groups:
\begin{align*}
    &(i) \, \, A_{21} + A_{12} =0 \text{ and } B_{11}=0, \\
    &(ii) \, \, \text{ the remaining structures.}
\end{align*}
For structures of type $(i)$, the behaviour of the spectral sequence stays the same, and we have degeneracy at the second stage. For structures of type $(ii)$ and $t \neq 0$, we have that the spectral sequence degenerates at the first stage and coincides with the spectral sequence of structure A and B. \\
In particular this shows that condition $^L E^{0,1}_1 = \, ^L E^{0,1}_2$ is not closed (at level of left-invariant spectral sequence), since it is not satisfied for $t=0$, but it is true for deformations of class $(ii)$.
\end{example}

\begin{example}[$\Gamma \backslash G$]\label{example:G} Let $G= \mathbb{R} \ltimes_\sigma \mathbb{R}^3$, with
\[
\sigma(t) = 
\begin{bmatrix}
e^{\alpha_2 t} & 0 & 0 \\
 & e^{\alpha_3 t} & 0 \\
 & & e^{\alpha_4 t}
\end{bmatrix},
\]
with $\alpha_j$ real numbers satisfying $\alpha_2 + \alpha_3 + \alpha_4 =0$. $G$ identifies as a subgroup of matrices, still denoted by $G$, via the homomorphism
\[
(t, x, y, z) \xmapsto{\quad  \quad}
\begin{bmatrix}
e^{\alpha_2 t} & 0 & 0 & x \\
 & e^{\alpha_3 t} & 0 & y \\
 & & e^{\alpha_4 t} & z \\
 & & & 1 \\
\end{bmatrix}
 \in SL(4, \mathbb{R}).
\]
The proof that $G$ admits a lattice $\Gamma$, and an explicit construction of the quotient can be found, as for example \ref{example:sol3}, in \cite{Bock}. Then the quotient $M = \Gamma \backslash G$ is a solvmanifold, of real dimension $4$. A basis of left-invariant forms is 
\[
\Big \{ e^1 = dt, e^2 = e^{-\alpha_2 t} dx, e^3 = e^{- \alpha_3 t} dy, e^4 = e^{-\alpha_4 t} dz \Big \}.
\]
The dual basis of vector fields is
\[
\Big \{ e_1 = \frac{\partial}{\partial t},e_2 = e^{\alpha_1 t} \frac{\partial}{\partial x}, e_3 = e^{\alpha_3 t} \frac{\partial}{\partial y}, e_4 = e^{\alpha_4 t} \frac{\partial}{\partial z} \Big \}.
\]
The only non-zero brackets are 
\[
[ e_1, e_j ] = \alpha_j e_j, \qquad j=2,3,4. 
\]
The real differentials are
\begin{align*}
\text{ $1$-forms: } \quad & de^1 =0, \quad de^2 = - \alpha_2 e^{12}, \quad de^3 = - \alpha_3 e^{13}, \quad de^4 = -\alpha_4 e^{14}; \\[1ex]
\text{ $2$-forms: } \quad & d e^{12} =0, \quad d e^{13} =0, \quad de^{14} =0, \quad de^{23}= \alpha_4 \, e^{123}, \\[1ex]
& de^{24} = \alpha_3 \, e^{124}, \quad de^{34} = \alpha_2 \,  e^{134}; \\[1ex]
\end{align*}
The differential vanishes on forms of degree $3$ and $4$. $G$ is completely solvable, and its left-invariant cohomology coincides with the real de Rham Cohomology:
\begin{align*}
    H^0_{dR} (M; \mathbb{R}) &= \langle 1 \rangle, \\
    H^1_{dR} (M; \mathbb{R}) &= \langle e^1 \rangle, \\
    H^2_{dR} (M; \mathbb{R}) &= \{ 0 \} \\
    H^3_{dR} (M; \mathbb{R}) &= \langle e^{234} \rangle, \\
    H^4_{dR} (M; \mathbb{R}) &= \langle e^{1234} \rangle. \\
\end{align*}
Betti numbers of $M$ are listed in Table \ref{G:betti}.
\begin{table}[H]
\centering
\[
\begin{array}{c|c|c|c|c|c}
       n & 0 & 1 & 2 & 3 & 4 \\
       \hline
       b^n &  1 & 1 & 0 & 1 & 1 \\
\end{array}
\]
\caption{Betti numbers for $\Gamma \backslash G$.}
\label{G:betti}
\end{table}
In particular, the Euler characteristic of $M$ is $0$. The complex de Rham cohomology of $M$ is obtained as the complexified of the real one. 
\newline 
We shall consider three almost complex structures on $M$. Set 
\[
\text{(A)}\quad  
\begin{cases}
Je_1 = e_2, \\
Je_3 = e_4, \\
\end{cases}
\qquad 
\text{(B)}\quad  
\begin{cases}
Je_1 = e_3, \\
Je_2 = e_4, \\
\end{cases}
\qquad 
\text{(C)}\quad  
\begin{cases}
Je_1 = e_4, \\
Je_2 = e_3, \\
\end{cases}
\]
For complex structure (A), set
\[
Z_1 = \frac{1}{2} ( e_1 - i e_2 ), \qquad Z_2 = \frac{1}{2} (e_3 -ie_4),
\]
and
\[
\varphi^1 = e^1 +ie^2, \qquad \varphi^2 = e^3 +i e^4.
\]
The  non-vanishing complex brackets are
\[
[Z_1, Z_2] = [\bar Z_1, Z_2] = \frac{\alpha_3+ \alpha_4}{4} \, Z_2 + \frac{\alpha_3 - \alpha_4}{4} \, \bar Z_2,
\]
\[
[\bar Z_1, \bar Z_2] = [ Z_1, \bar Z_2] = \frac{\alpha_3- \alpha_4}{4}\, Z_2 + \frac{\alpha_3 + \alpha_4}{4} \,\bar Z_2,
\]
\[ [Z_1, \bar Z_1] = -\frac{\alpha_2}{2}\, ( Z_1 - \bar Z_1),
\]
and the differentials
\begin{alignat*}{5}
&\text{ $1$-forms: } \quad &&\mu \varphi^1 = 0,  \quad &&\partial \varphi^1 =0, \quad &&\bar \partial \varphi^1 = \frac{\alpha_2}{2} \varphi^{1 \bar 1}, \quad &&\bar \mu \varphi^1 =0, 
\end{alignat*}
\[
\partial \varphi^2 =-\frac{\alpha_3 + \alpha_4}{4} \, \varphi^{12}, \qquad \bar \mu \varphi^2 =- \frac{\alpha_3 - \alpha_4}{4} \, \varphi^{\bar 1 \bar 2},
\]
\[
\mu \varphi^2 = 0, \qquad  \bar \partial \varphi^2 = - \frac{\alpha_3+\alpha_4}{4} \, \varphi^{\bar 1 2} - \frac{\alpha_3 - \alpha_4}{4} \, \varphi^{1 \bar 2 } ,
\]
and the conjugate equations. At degree $2$,
\begin{alignat*}{3}
&\text{ $2$-forms: } \quad &&\begin{cases}
\mu \varphi^{12} = \partial \varphi^{12} = 0, \\[2ex]
\bar \partial \varphi^{12} = -\frac{\alpha_2}{4} \, \varphi^{12 \bar1}, \\[2ex]
\bar \mu \varphi^{12} = \frac{\alpha_3 - \alpha_4}{4}\, \varphi^{1 \bar 1 \bar 2},\\[2ex]
\end{cases} \qquad \qquad 
&&\begin{cases}
\mu \varphi^{1\bar 2} = \bar \mu \varphi^{1 \bar 2} = 0, \\[2ex]
\partial \varphi^{1 \bar 2} = \frac{\alpha_4 - \alpha_3}{4} \, \varphi^{12 \bar1}, \\[2ex]
\bar \partial \varphi^{1 \bar 2} = \frac{\alpha_2}{4} \varphi^{1 \bar 1 \bar 2},\\[2ex]
\end{cases}
\\[2ex]
& &&\begin{cases}
\mu \varphi^{1 \bar 1} = \bar \mu \varphi^{1 \bar 1} = 0, \\[2ex]
\partial \varphi^{1 \bar 1} = \bar \partial \varphi^{1 \bar 1} =0 \\[2ex]
\end{cases} \qquad \qquad 
&&\begin{cases}
\mu \varphi^{2 \bar 2} = \bar \mu \varphi^{2 \bar 2} = 0, \\[2ex]
\partial \varphi^{2 \bar 2} = \frac{\alpha_2}{2} \, \varphi^{12 \bar 2}, \\[2ex]
\bar \partial \varphi^{2 \bar 2} = -\frac{\alpha_2}{2} \, \varphi^{2 \bar 1 \bar 2},\\[2ex]
\end{cases}
\end{alignat*}
Every equation has to be taken conjugate. On forms of degree $3$ and $4$, all differentials are $0$. The left-invariant Dolbeault cohomology is
\[
^L E^{*,*}_1 \cong 
\begin{array}{|c|c|c|} \hline
      \rule{0pt}{1em} 0 & 0 & \mathbb{C} \\ \hline
      \rule{0pt}{1em} \mathbb{C}^2 & \mathbb{C}^2 & \mathbb{C}^2 \\ \hline
      \rule{0pt}{1em} \mathbb{C} & 0 & 0 \\ \hline
\end{array}.
\]
The spectral sequence degenerates at the $2$-nd stage, giving the bigrading of the de Rham cohomology:
\[
^L E^{*,*}_r \cong 
\begin{array}{|c|c|c|} \hline
      \rule{0pt}{1em} 0 & 0 & \mathbb{C} \\ \hline
      \rule{0pt}{1em} \mathbb{C} & 0 & \mathbb{C} \\ \hline
      \rule{0pt}{1em} \mathbb{C} & 0 & 0 \\ \hline
\end{array}, \qquad \forall r \ge 2.
\]
If we consider complex structures (B) and (C), it just changes the pairing of vector fields, i.e., the constant $\alpha_j$ interchange with each other. For example, to compute the cohomology of (B), is enough to replace $\alpha_2$ with $\alpha_3$. The cohomologies obtained are thus isomorphic.
\end{example}

\vspace{1cm}

\small (Lorenzo Sillari)\\
\noindent \textsc{\small Scuola Internazionale Superiore di Studi Avanzati (SISSA), \\ Via Bonomea 265, 34136, Trieste, Italy.}
\newline
\small \emph{E-mail address:} lsillari@sissa.it

\vspace{1cm}

\small (Adriano Tomassini)\\
\noindent \textsc{\small Dipartimento di Scienze, Matematiche, Fisiche e Informatiche, \\ Unità di Matematica e Informatica, \\ Università degli Studi di Parma, \\ Parco Area delle Scienze 53/A, 43124, Parma, Italy.}
\newline
\small \emph{E-mail address:} adriano.tomassini@unipr.it

\end{document}